\documentclass{amsart}
\usepackage{amssymb}
\usepackage{amsfonts,color}
\usepackage{graphicx}
\usepackage{hyperref}
\newtheorem{theorem}{Theorem}[section]
\theoremstyle{plain}

\newtheorem{corollary}{Corollary}[section]

\newtheorem{definition}{Definition}[section]
\newtheorem{example}{Example}[section]

\newtheorem{lemma}{Lemma}[section]

\newtheorem{proposition}{Proposition}[section]
\newtheorem{remark}{Remark}[section]

\numberwithin{equation}{section}
\input{tcilatex}

\begin{document}
\title[block operator matrices]{some basic properties of block operator matrices}
\author[Jin]{Guohai Jin}
\address{Guohai Jin \\
\indent  School of Mathematical Sciences, Inner Mongolia University, Hohhot, 010021, China}\email{ghjin2006@gmail.com}
\author[Alatancang]{Alatancang Chen}
\address{Alatancang Chen \\
\indent  School of Mathematical Sciences, Inner Mongolia University, Hohhot, 010021, China;  Huhhot University for Nationalities, Hohhot, 010051, China}\email{alatanca@imu.edu.cn}
\thanks{Corresponding author: Alatancang Chen}

\begin{abstract}
General approach to the multiplication or adjoint operation of $2\times 2$ block operator matrices with unbounded entries are founded.
Furthermore, criteria for self-adjointness of block operator matrices based on their entry operators are established.
\end{abstract}
\maketitle

\section{introduction}
Block operator matrices arise in various areas of mathematical physics such as ordinary differential equations \cite{QC1, QC2, SS},
theory of elasticity \cite{LX,YZL, ZZ}, hydrodynamics \cite{KK2001,KK2003}, magnetohydrodynamics \cite{Lif1989},
quantum mechanics \cite{Tha}, and optimal control \cite{KZ, Wyss2011}.
The spectral properties of block operator matrices are of vital importance as they govern
for instance the solvability and stability of the underlying physical systems.
As a basis for spectral analysis,
the multiplication or adjoint operation
and self-adjointness of block operator matrices with unbounded entries
have attracted considerable attention and have been investigated case by case,
see, e.g, \cite{ALMS, Eng1998, Hassi, MS2008, Nag1989} for the former
and \cite{EL, FFMM, KLT, MS, Stra} for the latter,
or the monograph \cite{Tre} for both of these topics.
As has been pointed out in \cite{MS2008}, what is essentially trivial
for bounded operators appears to become erratic for unbounded operators.
The purpose of this paper is to build a common framework for these problems.

To this end, let us first recall some notions on block operator matrices.
Throughout this paper, we will denote by $X_1, X_2$ complex Banach spaces,
$X_1^*, X_2^*$ the adjoint spaces (see \cite[Section III.1.4]{Ka1980}),
and $X:=X_1\times X_2$ the product space equipped with the norm
$$\|(x_1\ x_2)^t\|:=(\|x_1\|^2+\|x_2\|^2)^{\frac{1}{2}}.$$
It is well known that $(X_1\times X_2)^*$ is isometrically isomorphic to $X_1^*\times X_2^*$ equipped with the norm
$$\|(f_1\ f_2)^t\|:=(\|f_1\|^2+\|f_2\|^2)^{\frac{1}{2}}$$
such that if the element $f$ of $(X_1\times X_2)^*$ is identified with the element $(f_1\ f_2)^t$ of $X_1^*\times X_2^*$, then
$$(f,x)=(f_1,x_1)+(f_2,x_2)$$
whenever $x=(x_1\ x_2)^t\in X_1\times X_2$ (see \cite[Theorem 1.10.13]{M1998}).
Following Engel \cite{Eng1998}, we define the injections $J_1, J_2$ and the projections $P_1, P_2$ as follows.
\begin{align*}
J_1:X_1\to X,\
J_1x_1:=\left(
         \begin{array}{c}
           x_1 \\
           0 \\
         \end{array}
       \right) \mbox{~and~}
P_1:X\to X_1,\
P_1\left(
     \begin{array}{c}
       x_1 \\
       x_2 \\
     \end{array}
   \right):=x_1,\\
J_2:X_2\to X,\
J_2x_2:=\left(
         \begin{array}{c}
           0 \\
           x_2 \\
         \end{array}
       \right) \mbox{~and~}
P_2:X\to X_2,\
P_2\left(
     \begin{array}{c}
       x_1 \\
       x_2 \\
     \end{array}
   \right):=x_2.
\end{align*}
Furthermore, we denote by $Q_1, Q_2$ the projections
\begin{align*}
Q_1:X\to X,\
Q_1\left(
     \begin{array}{c}
       x_1 \\
       x_2 \\
     \end{array}
   \right):=\left(
             \begin{array}{c}
               x_1 \\
               0 \\
             \end{array}
           \right),\\
Q_2:X\to X,\
Q_2\left(
     \begin{array}{c}
       x_1 \\
       x_2 \\
     \end{array}
   \right):=\left(
             \begin{array}{c}
               0 \\
               x_2 \\
             \end{array}
           \right).
\end{align*}

\begin{definition}(\cite[p. 97]{Wyss2008})
Let $A_{jk}:\mathcal D(A_{jk})\subset X_k\to X_j$ be linear operators, $j,k=1,2$.
The matrix
\begin{equation}\label{eqA1.1}
\mathcal A:=\left(
      \begin{array}{cc}
        A_{11} & A_{12} \\
        A_{21} & A_{22} \\
      \end{array}
    \right)
\end{equation}
is called a ($2\times 2$) block operator matrix on $X$.
It induces a linear operator on $X$ which is also denoted by $\mathcal A$:
\begin{align*}
\mathcal D(\mathcal A):&=(\mathcal D(A_{11})\cap\mathcal D(A_{21}))\times(\mathcal D(A_{12})\cap\mathcal D(A_{22})),\\
\mathcal A\left(
   \begin{array}{c}
     x_1 \\
     x_2 \\
   \end{array}
 \right):&=\left(
             \begin{array}{c}
               A_{11}x_1+A_{12}x_2 \\
               A_{21}x_1+A_{22}x_2 \\
             \end{array}
           \right),            \left(
                                  \begin{array}{c}
                                    x_1 \\
                                    x_2 \\
                                  \end{array}
                                \right)\in\mathcal D(\mathcal A).
\end{align*}
\end{definition}

\section{product and adjoint}
In this section, we shall establish rules for the product and adjoint operations of block operator matrices.

\begin{lemma}\label{lemA2.1}
Let $\mathcal A:\mathcal D(\mathcal A)\subset X\to X$ be a linear operator.
The following statements are equivalent.
\begin{enumerate}
\item[(a)] $\mathcal A$ has a matrix representation \eqref{eqA1.1},
\item[(b)] $\mathcal D(\mathcal A)= P_1\mathcal D(\mathcal A)\times P_2\mathcal D(\mathcal A)$,
\item[(c)] $Q_1\mathcal D(\mathcal A)\subset\mathcal D(\mathcal A)$ or, equivalently, $Q_2\mathcal D(\mathcal A)\subset\mathcal D(\mathcal A)$.
\end{enumerate}
Furthermore, if one of the above statements holds, then
\begin{equation}\label{eqA2.1}
\mathcal A=\left(
    \begin{array}{cc}
      P_1\mathcal AJ_1 & P_1\mathcal AJ_2 \\
      P_2\mathcal AJ_1 & P_2\mathcal AJ_2 \\
    \end{array}
  \right)
\end{equation}
in the sense of linear operators on $X$.
\end{lemma}
\begin{proof}
First, the statements (a) and (c) are equivalent (see \cite[p. 287]{TL1980}) and,
moreover, one see easily that the statements (b) and (c) are equivalent.
In addition, one readily checks that \eqref{eqA2.1} holds if $\mathcal A$ has a matrix representation, see also \cite{Eng1998}.
\end{proof}

\begin{definition}
Let $\mathcal A=(A_{jk}), \mathcal B=(B_{jk})$ be block operator matrices on $X$.
We define the formal product block operator matrix of $\mathcal A$ and $\mathcal B$ as follows:
$$\mathcal A\times\mathcal B:=(\sum\limits_{k=1}^2 A_{jk}B_{kl}).$$
\end{definition}

\begin{theorem}\label{thA2.1}
Let $\mathcal A, \mathcal B$ be block operator matrices on $X$.
Then $\mathcal A\times\mathcal B=\mathcal A\mathcal B$ if and only if
$\mathcal A\mathcal B$ has a matrix representation.
\end{theorem}
\begin{proof}
The ``only if" part is trivial.

The proof of the ``if" part. Assume $\mathcal A\mathcal B$ has a matrix representation.
Writing $\mathcal A=(A_{jk}), \mathcal B=(B_{jk})$
and $\mathcal D(\mathcal A)=\mathcal D_1\times\mathcal D_2$,
where $\mathcal D_k$ are subspaces of $X_k$ for $k=1,2$, respectively.
Then
\begin{align*}
&\mathcal D(\mathcal A\times\mathcal B)=\left\{\left(
      \begin{array}{c}
        x_1 \\
        x_2 \\
      \end{array}
    \right)\in\mathcal D(\mathcal B)~|~B_{11}x_1, B_{12}x_2\in\mathcal D_1, B_{21}x_1, B_{22}x_2\in\mathcal D_2\right\},\\
&\mathcal D(\mathcal A\mathcal B)=\left\{\left(
      \begin{array}{c}
        x_1 \\
        x_2 \\
      \end{array}
    \right)\in\mathcal D(\mathcal B)~|~B_{11}x_1+B_{12}x_2\in\mathcal D_1, B_{21}x_1+B_{22}x_2\in\mathcal D_2\right\},
\end{align*}
and so $\mathcal D(\mathcal A\times\mathcal B)\subset\mathcal D(\mathcal A\mathcal B)$.
One readily checks that $(\mathcal A\times\mathcal B)x=\mathcal A\mathcal B x$ for all $x\in\mathcal D(\mathcal A\times\mathcal B)$ which implies $\mathcal A\times\mathcal B\subset\mathcal A\mathcal B$.
It remains to show $\mathcal D(\mathcal A\mathcal B)\subset\mathcal D(\mathcal A\times\mathcal B)$.
If
$$x=\left(
       \begin{array}{c}
          x_1 \\
          x_2 \\
       \end{array}
    \right)\in\mathcal D(\mathcal A\mathcal B),$$
then by Lemma \ref{lemA2.1},
$$\left(
    \begin{array}{c}
      x_1 \\
      0 \\
    \end{array}
  \right), \left(
             \begin{array}{c}
               0 \\
               x_2 \\
             \end{array}
           \right)\in\mathcal D(\mathcal A\mathcal B),$$
and so from the structure of the set $\mathcal D(\mathcal A\mathcal B)$ we know that
\begin{align*}
&\left(
    \begin{array}{c}
      x_1 \\
      0 \\
    \end{array}
  \right)\in\mathcal D(\mathcal B),B_{11}x_1\in\mathcal D_1, B_{21}x_1\in\mathcal D_2,\\
&\left(
       \begin{array}{c}
               0 \\
               x_2 \\
             \end{array}
           \right)\in\mathcal D(\mathcal B),B_{12}x_2\in\mathcal D_1, B_{22}x_2\in\mathcal D_2,
\end{align*}
which imply $x\in\mathcal D(\mathcal A\times\mathcal B)$ by the structure of the set $\mathcal D(\mathcal A\times\mathcal B)$.
Hence $\mathcal D(\mathcal A\mathcal B)\subset\mathcal D(\mathcal A\times\mathcal B)$.
\end{proof}

The case $\mathcal A\mathcal B$ has no matrix representation can occur.

It follows from Theorem \ref{thA2.1} the \emph{Frobenius-Schur fractorization} of a block operator matrix.
\begin{corollary}(\cite[Section 2.2]{Tre})\label{corA2.1}
Let
\begin{align*}
\mathcal A:=\left(
    \begin{array}{cc}
      A & B \\
      C & D \\
    \end{array}
  \right)
\end{align*}
be a block operator matrix acting on $X$.
\begin{enumerate}
\item[(a)] Suppose that $D$ is closed with $\rho(D)\neq\emptyset$, and that $\mathcal D(D)\subset\mathcal D(B)$.
Then for some (and hence for all) $\lambda\in\rho(D)$,
\begin{align*}
\mathcal A-\lambda=&\left(
            \begin{array}{cc}
              I & B(D-\lambda)^{-1} \\
              0 & I \\
            \end{array}
          \right)\left(
                   \begin{array}{cc}
                     S_1(\lambda) & 0 \\
                     0 & D-\lambda \\
                   \end{array}
                 \right)\\
                 &\left(
                          \begin{array}{cc}
                            I & 0 \\
                            (D-\lambda)^{-1}C & I \\
                          \end{array}
                        \right),
\end{align*}
where $S_1(\lambda):=A-\lambda-B(D-\lambda)^{-1}C$ is the first Schur complement of $\mathcal A$ with domain $\mathcal D(S_1(\lambda))=\mathcal D(A)\cap\mathcal D(C)$.
\item[(b)] Suppose that $A$ is closed with $\rho(A)\neq\emptyset$, and that $\mathcal D(A)\subset\mathcal D(C)$.
Then for some (and hence for all) $\lambda\in\rho(A)$,
\begin{align*}
\mathcal A-\lambda=&\left(
            \begin{array}{cc}
              I & 0 \\
              C(A-\lambda)^{-1} & I \\
            \end{array}
          \right)\left(
                   \begin{array}{cc}
                     A-\lambda & 0 \\
                     0 & S_2(\lambda) \\
                   \end{array}
                 \right)\\
                 &\left(
                          \begin{array}{cc}
                            I & (A-\lambda)^{-1}B \\
                            0 & I \\
                          \end{array}
                        \right),
\end{align*}
where $S_2(\lambda):=D-\lambda-C(A-\lambda)^{-1}B$ is the second Schur complement of $\mathcal A$ with domain $\mathcal D(S_2(\lambda))=\mathcal D(B)\cap\mathcal D(D)$.
\end{enumerate}
\end{corollary}
\begin{proof}
To prove the first equality,
we denote by $\mathcal R\mathcal S\mathcal T$ the product of the three linear operators on the right side.
It is easy to see that
$$\mathcal D(\mathcal R\mathcal S\mathcal T)=\mathcal D(\mathcal S\mathcal T)=(\mathcal D(A)\cap\mathcal D(C))\times\mathcal D(D).$$
Hence, we have, by Theorem \ref{thA2.1},
$$\mathcal R\mathcal S\mathcal T=\mathcal R\times(\mathcal S\times\mathcal T),$$
this proved the first equality.
Similarly, the second equality holds.
\end{proof}

\begin{definition}
Let
$$\mathcal A=\left(
                \begin{array}{cc}
                  A & B \\
                  C & D \\
                \end{array}
              \right)$$
be a block operator matrix on $X$ with dense domain $\mathcal D_1\times\mathcal D_2$.
Then the block operator matrix
$$\mathcal A^\times:=\left(
                  \begin{array}{cc}
                    (A|_{\mathcal D_1})^* & (C|_{\mathcal D_1})^* \\
                    (B|_{\mathcal D_2})^* & (D|_{\mathcal D_2})^* \\
                  \end{array}
                \right)$$
is said to be the formal adjoint block operator matrix of $\mathcal A$.
\end{definition}

\begin{theorem}\label{thA2.2}
Let $\mathcal A$ be a block operator matrix on $X$ with dense domain.
Then $\mathcal A^\times=\mathcal A^*$ if and only if $\mathcal A^*$ has a matrix representation.
\end{theorem}

\begin{proof}
We only need to prove the ``if" part. Assume $\mathcal A^*$ has a matrix representation.
Writing
$$\mathcal A=\left(
                   \begin{array}{cc}
                     A & B \\
                     C & D \\
                   \end{array}
                 \right)$$
and $\mathcal D(\mathcal A)=\mathcal D_1\times\mathcal D_2$.
Then it follows from
$(f,\mathcal A x)=(\mathcal A^\times f, x)$ for all $x\in\mathcal D(\mathcal A)$ and all $f\in\mathcal D(\mathcal A^\times)$
that $\mathcal A^\times\subset\mathcal A^*$ (see also \cite{MS2008}).
It remains to show $\mathcal D(\mathcal A^*)\subset\mathcal D(\mathcal A^\times)$.
Let
$$g=\left(
\begin{array}{c}
g_1 \\
g_2 \\
\end{array}
\right)\in \mathcal D(\mathcal A^*).$$
By Lemma \ref{lemA2.1} we have $Q_kg\in\mathcal D(\mathcal A^*)$ for $k=1, 2$,
so that
\begin{equation}\label{eqA2.2}
(Q_kg, \mathcal A y)=(\mathcal A^* Q_kg, y) \mbox{~for~} y\in\mathcal D(\mathcal A), k=1,2.
\end{equation}
Writing
$$\mathcal A^*Q_kg=\left(
            \begin{array}{c}
              h_{k1} \\
              h_{k2} \\
            \end{array}
          \right), k=1,2.$$
By taking $k=1$ in \eqref{eqA2.2} we get
\begin{equation*}
(g_1,A|_{\mathcal D_1}y_1)+(g_1,B|_{\mathcal D_1}y_2)=(h_{11},y_1)+(h_{12},y_2),
          \left(
            \begin{array}{c}
              y_1 \\
              y_2 \\
            \end{array}
          \right)\in \mathcal D(\mathcal A).
\end{equation*}
It follows that
\begin{align*}
(g_1,A|_{\mathcal D_1}y_1)=(h_{11},y_1) \mbox{~for~} y_1\in \mathcal D(A|_{\mathcal D_1}),\\
(g_1,B|_{\mathcal D_1}y_2)=(h_{12},y_2) \mbox{~for~} y_2\in \mathcal D(B|_{\mathcal D_1}),
\end{align*}
so that
\begin{align}\label{eqA2.3}
g_1\in \mathcal D((A|_{\mathcal D_1})^*)\cap\mathcal D((B|_{\mathcal D_1})^*).
\end{align}
Similarly, by taking $k=2$ in \eqref{eqA2.2} we get
\begin{align}\label{eqA2.4}
g_2\in \mathcal D((C|_{\mathcal D_1})^*)\cap\mathcal D((D|_{\mathcal D_1})^*).
\end{align}
From \eqref{eqA2.3} and \eqref{eqA2.4} we get
$g\in\mathcal D(\mathcal A^\times)$,
so that $\mathcal D(\mathcal A^*)\subset\mathcal D(\mathcal A^\times)$.
\end{proof}

Obviously, $\left(
       \begin{array}{cc}
         A^* & C^* \\
         B^* & D^* \\
       \end{array}
     \right)
\subset\mathcal A^\times\subset\mathcal A^*$.
Moreover, there is a block operator matrix $\mathcal A$ such that
$\mathcal A^\times=\mathcal A^*$
but $\left(
       \begin{array}{cc}
         A^* & C^* \\
         B^* & D^* \\
       \end{array}
     \right)
\neq\mathcal A^*$,
see the following example (for the notions of differential operators see \cite{Nai1968}).

\begin{example}\label{exA2.1}
Given the following four differential operators on the Hilbert space $L^2(0,1)$:
\begin{align*}
&\mathcal D(L):=H^2(0,1), Lf:=-f'',\\
&\mathcal D(L_0):=\{f\in\mathcal D(L)~|~f^{(k)}(0)=f^{(k)}(1)=0, k=0, 1\}, L_0:=L|_{\mathcal D(L_0)},\\
&\mathcal D(L_D):=\{f\in\mathcal D(L)~|~f(0)=f(1)=0\}, L_D:=L|_{\mathcal D(L_D)},\\
&\mathcal D(L_N):=\{f\in\mathcal D(L)~|~f'(0)=f'(1)=0\}, L_N:=L|_{\mathcal D(L_N)}.
\end{align*}

It is well known that $L_0$ is closed, $L=L_0^*, L_D=L_D^*$, and $L_N=L_N^*$.
For the block operator matrix
$$\mathcal L:=\left(
      \begin{array}{cc}
        L_D & L_N \\
        L_N & -L_D \\
      \end{array}
    \right)$$
defined on the Hilbert space $L^2(0,1)\times L^2(0,1)$, we have
\begin{enumerate}
\item[(a)] $\mathcal L$ is closed,
\item[(b)] $\mathcal L=\left(
      \begin{array}{cc}
        L_D^* & L_N^* \\
        L_N^* & -L_D^* \\
      \end{array}
    \right)\subsetneq\mathcal L^\times=\mathcal L^*$.
\end{enumerate}
In fact, we see from $\mathcal D(L_0)=\mathcal D(L_D)\cap\mathcal D(L_N)$ that
$$\mathcal L=\left(
      \begin{array}{cc}
        L_0 & L_0 \\
        L_0 & -L_0 \\
      \end{array}
    \right)=\left(
        \begin{array}{cc}
          I & I \\
          I & -I \\
        \end{array}
      \right)\left(
               \begin{array}{cc}
                 L_0 & 0 \\
                 0 & L_0 \\
               \end{array}
             \right),$$
where the latter equality follows from Theorem \ref{thA2.1}.
By Lemma \ref{lemA1},
\begin{equation}\label{eqA2.5}
\mathcal L^*=\left(
               \begin{array}{cc}
                 L & 0 \\
                 0 & L \\
               \end{array}
             \right)\left(
        \begin{array}{cc}
          I & I \\
          I & -I \\
        \end{array}
      \right),
\end{equation}
which implies $\mathcal D(\mathcal L^*)=\mathcal D(L)\times\mathcal D(L)$.
Thus, by Theorem \ref{thA2.1},
$$\mathcal L^*=\left(
      \begin{array}{cc}
        L & L \\
        L & -L \\
      \end{array}
    \right)=\mathcal L^\times.$$
But
$$\left(
      \begin{array}{cc}
        L_D^* & L_N^* \\
        L_N^* & -L_D^* \\
      \end{array}
    \right)=\mathcal L\subsetneq\mathcal L^*.$$
It remains to prove $\mathcal L$ is closed.
Since
$\frac{1}{\sqrt{2}}\left(
   \begin{array}{cc}
     I & I \\
     I & -I \\
   \end{array}
 \right)$ is a self-adjoint unitary operator, we have, by applying Lemma \ref{lemA2} to \eqref{eqA2.5},
$$\mathcal L^{**}=\left(
   \begin{array}{cc}
     I & I \\
     I & -I \\
   \end{array}
 \right)\left(
               \begin{array}{cc}
                 L^* & 0 \\
                 0 & L^* \\
               \end{array}
             \right)=\mathcal L.$$
\end{example}

If a densely defined operator $\mathcal A$ has a matrix representation,
this need not be true for $\mathcal A^*$ even if $\mathcal A$ is closed; see the following example.
\begin{example}
Let $X_1=X_2$ and let $A$ be a closed densely defined operator on $X_1$ with $\mathcal D(A)\neq X_1$.
Consider the block operator matrix
$$\mathcal A:=\left(
      \begin{array}{cc}
        A & 0 \\
        A & 0 \\
      \end{array}
    \right)$$
on $X_1\times X_1$. It is easy to verify that $\mathcal A$ is a closed operator
with domain $\mathcal D(\mathcal A)=\mathcal D(A)\times X_1$.
Furthermore, we see from Theorem \ref{thA2.1} that
$$\mathcal A=\left(
      \begin{array}{cc}
        I & 0 \\
        I & 0 \\
      \end{array}
    \right)\left(
             \begin{array}{cc}
               A & 0 \\
               0 & 0 \\
             \end{array}
           \right).$$
By Lemma \ref{lemA1},
$$\mathcal A^*=\left(
             \begin{array}{cc}
               A^* & 0 \\
               0 & 0 \\
             \end{array}
           \right)\left(
      \begin{array}{cc}
        I & I \\
        0 & 0 \\
      \end{array}
    \right),$$
so that
\begin{align*}
\mathcal D(\mathcal A^*)=\left\{\left(
      \begin{array}{c}
        x_1 \\
        x_2 \\
      \end{array}
    \right)\in X_1\times X_1~|~x_1+x_2\in\mathcal D(A^*)\right\}.
\end{align*}
Since $A$ is unbounded, we have $\mathcal D(A^*)\neq X_1$.
Taking $x_1\in X_1\setminus\mathcal D(A^*)$, then
$$\left(
    \begin{array}{c}
      x_1 \\
      -x_1 \\
    \end{array}
  \right)\in\mathcal D(\mathcal A^*)\mbox{~but~}
  \left(
    \begin{array}{c}
      x_1 \\
      0 \\
    \end{array}
  \right)\not\in\mathcal D(\mathcal A^*).$$
Hence, by Lemma \ref{lemA2.1}, $\mathcal A^*$ has no block operator matrix representation.
\end{example}

\section{self-adjointness}
Let $H_1, H_2$ be complex Hilbert spaces.
Now we consider self-adjointness of block operator matrices with unbounded entries acting on the Hilbert space $H_1\times H_2$.

First we shall consider necessary conditions for a block operator matrix to be self-adjoint.
Let
$\mathcal A:=\left(
   \begin{array}{cc}
     A & B \\
     C & D \\
   \end{array}
 \right)$
be a block operator matrix acting on $H_1\times H_2$ with dense domain $\mathcal D_1\times\mathcal D_2$.
Clearly, $\mathcal A$ is symmetric if and only if
\begin{equation}\label{eqA3.1}
A|_{\mathcal D_1}\subset(A|_{\mathcal D_1})^*, B|_{\mathcal D_2}\subset(C|_{\mathcal D_1})^*, D|_{\mathcal D_2}\subset(D|_{\mathcal D_2})^*.
\end{equation}
Furthermore, it follows from Theorem \ref{thA2.2} and \eqref{eqA3.1} the following assertion.
\begin{proposition}
If $\mathcal A$ is self-adjoint, then
$$\mathcal A=\left(
                \begin{array}{cc}
                  \overline{A|_{\mathcal D_1}} & \overline{B|_{\mathcal D_2}} \\
                  \overline{C|_{\mathcal D_1}} & \overline{D|_{\mathcal D_2}} \\
                \end{array}
              \right)=\mathcal A^\times.$$
\end{proposition}

In addition,
we point out that the entry operators of a self-adjoint block operator matrix need not be closed, see the following example.

\begin{example}\label{exmA3.1}
Let $X:=(L^2(0,1)\times L^2(0,1))\times (L^2(0,1)\times L^2(0,1))$
Let $M$ be the differential operator on the Hilbert space $L^2(0,1)$ which is defined by
$$\mathcal D(M):=\{f\in H^1(0,1)~|~f(0)=f(1)=0\}, Mf:=if'.$$
From the methods of differential operators we know that $M$ is closed, $C_c^\infty(0,1)$ is a core of $M$,
and $M^*$ is determined by
$$\mathcal D(M^*):=H^1(0,1), M^*f:=if'.$$
Let $L_D$ be the same as in Example \ref{exA2.1} and let $M_0:=M|_{\mathcal D(L_D)}$.
For the block operator matrix
$$\mathcal A:=\left(
      \begin{array}{cccc}
        L_D & 0 & M_0 & 0 \\
        0 & M_0 & 0 & L_D \\
        M_0 & 0 & L_D & 0 \\
        0 & L_D & 0 & M_0 \\
      \end{array}
    \right)=:\left(
               \begin{array}{cc}
                 A & B \\
                 B & A \\
               \end{array}
             \right)$$
on the Hilbert space $X$, we claim that
\begin{enumerate}
\item[(a)] $\mathcal A$ is self-adjoint,
\item[(b)] $A, B$ are not closed and $A\subsetneq A^*, B\subsetneq B^*$.
\end{enumerate}
In fact, it is easy to see that
$$\mathcal A=\left(
      \begin{array}{cccc}
        I & 0 & 0 & 0 \\
        0 & 0 & I & 0 \\
        0 & I & 0 & 0 \\
        0 & 0 & 0 & I \\
      \end{array}
    \right)\left(
             \begin{array}{cccc}
               L_D & M_0 & 0 & 0 \\
               M_0 & L_D & 0 & 0 \\
               0 & 0 & M_0 & L_D \\
               0 & 0 & L_D & M_0 \\
             \end{array}
           \right)\left(
      \begin{array}{cccc}
        I & 0 & 0 & 0 \\
        0 & 0 & I & 0 \\
        0 & I & 0 & 0 \\
        0 & 0 & 0 & I \\
      \end{array}
    \right)=:\mathcal E\mathcal B\mathcal E.$$
By interpolation theorem of Sobolev spaces (see \cite[Theorem 5.2]{AF2003}) and Lemma \ref{lemA3}, the operators
$\left(
             \begin{array}{cc}
               M_0 & L_D \\
               L_D & M_0 \\
             \end{array}
           \right)\mbox{~and~}
\left(
    \begin{array}{cc}
      L_D & M_0 \\
      M_0 & L_D \\
    \end{array}
  \right)$
are self-adjoint. Hence, $\mathcal B$ is self-adjoint (see \cite[Proposition 2.6.3]{Tre}).
Since $\mathcal E^*=\mathcal E^{-1}=\mathcal E$,
it follows from Lemma \ref{lemA2} and Lemma \ref{lemA3} that $\mathcal A$ is self-adjoint.
This proved the first claim.
The second claim follows from the following four equalities:
\begin{align*}
A^*=\left(
      \begin{array}{cc}
        L_D & 0 \\
        0 & M^* \\
      \end{array}
    \right), \ \overline{A}=\left(
                            \begin{array}{cc}
                              L_D & 0 \\
                              0 & M \\
                            \end{array}
                          \right),\\
B^*=\left(
      \begin{array}{cc}
        M^* & 0 \\
        0 & L_D \\
      \end{array}
    \right), \ \overline{B}=\left(
                            \begin{array}{cc}
                              M & 0 \\
                              0 & L_D \\
                            \end{array}
                          \right).
\end{align*}
\end{example}

Next we consider sufficient conditions for a block operator matrix to be (essentially) self-adjoint.
In view of \eqref{eqA3.1} and Example \ref{exmA3.1},
through out the rest of this section we make the following \emph{basic assumptions}:
\begin{enumerate}
\item[(i)] $A, B, C, D$ are densely defined and closable,
\item[(ii)] $A\subset A^*, B\subset C^*, D\subset D^*$,
\item[(iii)] $\mathcal A:=\left(
                \begin{array}{cc}
                  A & B \\
                  C & D \\
                \end{array}
              \right)$ is densely defined on $H_1\times H_2$.
\end{enumerate}
Further assumptions will be formulated where they are needed.

\begin{proposition}
$\mathcal A$ is self-adjoint if one of the following statements holds:
\begin{enumerate}
\item[(a)] $A, D$ are self-adjoint, $C$ is $A$-bounded with relative bound $<1$, and $B$ is $D$-bounded with relative bound $<1$.
\item[(b)] $B$ is closed, $C=B^*$, $A$ is $C$-bounded with relative bound $<1$, and $D$ is $B$-bounded with relative bound $<1$.
\end{enumerate}
\end{proposition}
\begin{proof}
We prove the claim in case (a); the proof in case (b) is analogous.
Writing
$$\mathcal A=\left(
                \begin{array}{cc}
                  A & 0 \\
                  0 & D \\
                \end{array}
              \right)+\left(
                        \begin{array}{cc}
                          0 & B \\
                          C & 0 \\
                        \end{array}
                      \right)=:\mathcal S+\mathcal T.$$
Then $\mathcal S$ is self-adjoint and $\mathcal T$ is symmetric (see \cite[Proposition 2.6.3]{Tre}).
Furthermore, by the assumptions $\mathcal T$ is $\mathcal S$-bounded with relative bound $<1$.
By applying Lemma \ref{lemA3} to $\mathcal S, \mathcal T$, we complete the proof.
\end{proof}

\begin{proposition}
$\mathcal A$ is essentially self-adjoint if one of the following statements holds:
\begin{enumerate}
\item[(a)] $A, D$ are self-adjoint, and for some $a\geq 0$,
\begin{align*}
\|Cx\|^2\leq a\|x\|^2+\|Ax\|^2, x\in\mathcal D(A),\\
\|By\|^2\leq a\|y\|^2+\|Dy\|^2, y\in\mathcal D(D).
\end{align*}
\item[(b)] $B$ is closed, $C=B^*$, and for some $a\geq 0$,
\begin{align*}
\|Ax\|^2\leq a\|x\|^2+\|Cx\|^2, x\in\mathcal D(C),\\
\|Dy\|^2\leq a\|y\|^2+\|By\|^2, y\in\mathcal D(B).
\end{align*}
\end{enumerate}
\end{proposition}
\begin{proof}
We prove e.g. case (a).
By the assumptions, we have, for all $(x,\ y)^t\in\mathcal D(A)\times\mathcal D(D)$,
$$\left\|\left(
      \begin{array}{cc}
        0 & B \\
        C & 0 \\
      \end{array}
    \right)\left(
             \begin{array}{c}
               x \\
               y \\
             \end{array}
           \right)
\right\|\leq \sqrt{a}\left\|\left(
                   \begin{array}{c}
                     x \\
                     y \\
                   \end{array}
                 \right)
\right\|+\left\|\left(
                        \begin{array}{cc}
                          A & 0 \\
                          0 & D \\
                        \end{array}
                      \right)\left(
                               \begin{array}{c}
                                 x \\
                                 y \\
                               \end{array}
                             \right)
\right\|.$$
Consequently, the assertion follows from the W\"{u}st theorem (see \cite[Theorem 4]{W1972}).
\end{proof}

\begin{theorem}\label{thA3.1}
The following statements hold.
\begin{enumerate}
\item[(a)] If $D=D^*, \mathcal D(D)\subset\mathcal D(B)$, then $\mathcal A$ is self-adjoint if and only if $$(A-B(D-\lambda)^{-1}C)^*=A-B(D-\overline{\lambda})^{-1}C$$
for some (and hence for all) $\lambda\in\rho(D)$.
\item[(b)] If $A=A^*, \mathcal D(A)\subset\mathcal D(C)$, then $\mathcal A$ is self-adjoint if and only if $$(D-C(A-\lambda)^{-1}B)^*=D-C(A-\overline{\lambda})^{-1}B$$
for some (and hence for all) $\lambda\in\rho(A)$.
\end{enumerate}
\end{theorem}
\begin{proof}
Proof of (a). Let $\lambda\in\rho(D)$.
By applying Corollary \ref{corA2.1} to $(\mathcal A-\lambda)$ we have
\begin{align}\label{eqA3.2}
\mathcal A-\lambda=&\left(
            \begin{array}{cc}
              I & B(D-\lambda)^{-1} \\
              0 & I \\
            \end{array}
          \right)\left(
                   \begin{array}{cc}
                     S_1(\lambda) & 0 \\
                     0 & D-\lambda \\
                   \end{array}
                 \right)\\
\nonumber                &\left(
                          \begin{array}{cc}
                            I & 0 \\
                            (D-\lambda)^{-1}C & I \\
                          \end{array}
                        \right),
\end{align}
where $S_1(\lambda):=A-\lambda-B(D-\lambda)^{-1}C$.
We see from $\mathcal D(D)\subset\mathcal D(B), B\subset C^*$ that $\mathcal D(D)\subset\mathcal D(C^*)$,
so that $(D-\lambda)^{-1}C$ is bounded on its domain $\mathcal D(C)$ (see \cite[Proposition 3.1]{ALMS}).
Moreover, in \eqref{eqA3.2}, the domain of the middle factor is equal to $(\mathcal D(A)\cap\mathcal D(C))\times\mathcal D(D)$,
and so we can replace
$(D-\lambda)^{-1}C=(\overline{(D-\lambda)^{-1}C})|_{\mathcal D(C)}$ by
\begin{align*}
\overline{(D-\lambda)^{-1}C}&=((D-\lambda)^{-1}C)^{**}\\
                              &=(C^*(D-\overline{\lambda})^{-1})^*\mbox{\ \ \ \ \ (by Lemma \ref{lemA1})}\\
                              &=(B(D-\overline{\lambda})^{-1})^*.
\end{align*}
It follows that
\begin{align}\label{eqA3.3}
\mathcal A-\lambda=&\left(
            \begin{array}{cc}
              I & B(D-\lambda)^{-1} \\
              0 & I \\
            \end{array}
          \right)\left(
                   \begin{array}{cc}
                     S_1(\lambda) & 0 \\
                     0 & D-\lambda \\
                   \end{array}
                 \right)\\
\nonumber                &\left(
                          \begin{array}{cc}
                            I & 0 \\
                            (B(D-\overline{\lambda})^{-1})^* & I \\
                          \end{array}
                        \right).
\end{align}
In the factorization \eqref{eqA3.3}, the first and last factor are bounded and boundedly invertible,
and therefore by Lemma \ref{lemA1} and Lemma \ref{lemA2},
\begin{align}\label{eqA3.4}
\mathcal A^*-\overline{\lambda}=&\left(
            \begin{array}{cc}
              I & B(D-\overline{\lambda})^{-1} \\
              0 & I \\
            \end{array}
          \right)\left(
                   \begin{array}{cc}
                     S_1(\lambda)^* & 0 \\
                     0 & D-\overline{\lambda} \\
                   \end{array}
                 \right)\\
\nonumber                &\left(
                          \begin{array}{cc}
                            I & 0 \\
                            (B(D-\lambda)^{-1})^* & I \\
                          \end{array}
                        \right).
\end{align}
Furthermore, in \eqref{eqA3.3} we can replace $\lambda$ by $\overline{\lambda}$ and obtain
\begin{align}\label{eqA3.5}
\mathcal A-\overline{\lambda}=&\left(
            \begin{array}{cc}
              I & B(D-\overline{\lambda})^{-1} \\
              0 & I \\
            \end{array}
          \right)\left(
                   \begin{array}{cc}
                     S_1(\overline{\lambda}) & 0 \\
                     0 & D-\overline{\lambda} \\
                   \end{array}
                 \right)\\
\nonumber                &\left(
                          \begin{array}{cc}
                            I & 0 \\
                            (B(D-\lambda)^{-1})^* & I \\
                          \end{array}
                        \right).
\end{align}
We conclude from \eqref{eqA3.4},\eqref{eqA3.5} that $\mathcal A^*=\mathcal A$ if and only if $S_1(\lambda)^*=S_1(\overline{\lambda})$.

Proof of (b). Let $\lambda\in\rho(A)$. Similar to the proof of (a) we have $\mathcal A^*=\mathcal A$ if and only if $S_2(\lambda)^*=S_2(\overline{\lambda})$,
where $S_2(\lambda):=D-\lambda-C(A-\lambda)^{-1}B$.
\end{proof}

\begin{corollary}
Let $A=A^*, D=D^*$.
Then $\mathcal A$ is self-adjoint if one of the following holds:
\begin{enumerate}
\item[(a)] $C$ is $A$-bounded with relative bound $< 1$ and $B$ is $D$-bounded with relative bound $\leq 1$,
\item[(b)] $C$ is $A$-bounded with relative bound $\leq 1$ and $B$ is $D$-bounded with relative bound $< 1$.
\end{enumerate}
\end{corollary}
\begin{proof}
We prove the claim in case (a); the proof in case (b) is analogous.
It is enough to prove
\begin{equation}\label{eqA3.6}
(A-B(D-i\lambda)^{-1}C)^*=A-B(D+i\lambda)^{-1}C \mbox{~for some~} \lambda>0.
\end{equation}

\noindent \emph{Step 1}. We start from
the claim that for $\lambda>0$ large enough, $B(D-i\lambda)^{-1}C$ is $A$-bounded with relative bound $<1$.
Since $C$ is $A$-bounded with relative bound $< 1$, it is enough to prove for each $\varepsilon>0$ there exists $\lambda>0$ such that
\begin{equation*}
\|B(D-i\lambda)^{-1}\|<1+\varepsilon.
\end{equation*}
We observe that for $x\in\mathcal D(D)$ and $\lambda>0$,
\begin{equation}\label{eqA3.7}
\|(D-i\lambda)x\|^2=\|Dx\|^2+\lambda^2\|x\|^2
\end{equation}
since $D$ is self-adjoint.
By the assumption that $B$ is $D$-bounded with relative bound $\leq 1$,
there exists $a(\varepsilon)\geq 0$ such that
\begin{equation}\label{eqA3.8}
\|Bx\|\leq (1+\frac{\varepsilon}{2})\|Dx\|+a(\varepsilon)\|x\|, x\in\mathcal D(D),
\end{equation}
so that for $x\in\mathcal D(D)$, we have, using \eqref{eqA3.7} twice and then \eqref{eqA3.8},
\begin{equation*}
\|Bx\|\leq (1+\frac{\varepsilon}{2}+\frac{a(\varepsilon)}{\lambda})\|(D-i\lambda)x\|.
\end{equation*}
It is enough to choose $\lambda>0$ large enough such that $\frac{a(\varepsilon)}{\lambda}<\frac{\varepsilon}{2}$.

\noindent \emph{Step 2}. In this step, we show that for $\lambda>0$ large enough, $(B(D-i\lambda)^{-1}C)^*$ is $A$-bounded with relative bound $<1$.
Since $\mathcal D(D)\subset\mathcal D(B)$ and $B$ is closable, $B(D-i\lambda)^{-1}$ is everywhere defined and closed,
so that it is bounded by the closed graph theorem.
Thus, by Lemma \ref{lemA1},
\begin{equation*}
(B(D-i\lambda)^{-1}C)^*=C^*(B(D-i\lambda)^{-1})^*=C^*\overline{(D+i\lambda)^{-1}B^*}.
\end{equation*}
Now $(D+i\lambda)^{-1}B^*$ is bounded on $\mathcal D(B^*)$ since $\mathcal D(B)\supset\mathcal D(D)$ (see \cite[Proposition 3.1]{ALMS}),
so that $(C^*\overline{(D+i\lambda)^{-1}B^*})|_{\mathcal D(B^*)}=C^*(D+i\lambda)^{-1}B^*=B(D+i\lambda)^{-1}B^*$.
Then it follows from $\mathcal D(A)\subset\mathcal D(C), C\subset B^*$ that
\begin{equation}\label{eqA3.9}
(B(D-i\lambda)^{-1}C)^*|_{\mathcal D(A)}=(B(D+i\lambda)^{-1}C)|_{\mathcal D(A)}.
\end{equation}
Thus, by \emph{Step 1}, $(B(D-i\lambda)^{-1}C)^*$ is $A$-bounded with relative bound $<1$.

\noindent \emph{Step 3}. Now \eqref{eqA3.6} follows from \emph{Step 1} and \emph{Step 2}
by applying Lemma \ref{lemA3} and \eqref{eqA3.9}.
\end{proof}

\begin{corollary}
Let $A=A^*, D=D^*$.
Then $\mathcal A$ is self-adjoint if one of the following holds:
\begin{enumerate}
\item[(a)] $C$ is $A$-bounded with relative bound $0$, and $\mathcal D(D)\subset\mathcal D(B)$.
\item[(b)] $\mathcal D(A)\subset\mathcal D(C)$, and $B$ is $D$-bounded with relative bound $0$.
\end{enumerate}
\end{corollary}

\begin{proof}
We prove the claim in case (a); the proof in case (b) is analogous.
Let $\lambda\in\rho(A)$. We need to prove
\begin{equation}\label{eqA3.10}
(D-C(A-\lambda)^{-1}B)^*=D-C(A-\overline{\lambda})^{-1}B.
\end{equation}
\noindent \emph{Step 1}. First we claim that $C(A-\lambda)^{-1}B$ is $D$-bounded with relative bound $0$.
Since $C$ is $A$-bounded with relative bound $0$,
for each $\varepsilon>0$, there exists $b_1(\varepsilon,\lambda)\geq 0$, such that
$$\|C(A-\lambda)^{-1}x\|\leq\varepsilon\|x\|+b_1(\varepsilon,\lambda)\|(A-\lambda)^{-1}x\|, x\in H_1,$$
so that for $x\in\mathcal D(B)$,
\begin{align*}
\|C(A-\lambda)^{-1}Bx\|&\leq\varepsilon\|Bx\|+b_1(\varepsilon,\lambda)\|(A-\lambda)^{-1}Bx\| \\
                       &\leq\varepsilon\|Bx\|+b_2(\varepsilon,\lambda)\|x\|,
\end{align*}
where the last inequality follows from the fact that $(A-\lambda)^{-1}B$ is bounded on $\mathcal D(B)$
(since $\mathcal D(B^*)\supset\mathcal D(C)\supset\mathcal D(A)$).
Furthermore, since $D$ is closed and $\mathcal D(D)\subset\mathcal D(B)$, there are $a, b\geq 0$ such that
$$\|Bx\|\leq a\|Dx\|+b\|x\|, x\in\mathcal D(D),$$
so that
$$\|C(A-\lambda)^{-1}Bx\|\leq\varepsilon a\|Dx\|+b_3(\varepsilon,\lambda)\|x\|, x\in\mathcal D(D).$$

\noindent \emph{Step 2}. We have, with arguments similar to the ones used in the proof of \eqref{eqA3.9},
$$(C(A-\lambda)^{-1}B)^*|_{\mathcal D(D)}=C(A-\overline{\lambda})^{-1}B|_{\mathcal D(D)},$$
so that by \emph{Step 1}, $(C(A-\lambda)^{-1}B)^*$ is $D$-bounded with relative bound $0$.

\noindent \emph{Step 3}. Finally, \eqref{eqA3.10} follows from \emph{Step 1} and \emph{Step 2}
by applying Lemma \ref{lemA3}.
\end{proof}

The following analogue of Theorem \ref{thA3.1} can be proved in the same way.
\begin{theorem}
The following statements hold.
\begin{enumerate}
\item[(a)] If $D=D^*$ with $\mathcal D(D)\subset\mathcal D(B)$, then $\mathcal A$ is essentially self-adjoint if and only if $$(A-B(D-\lambda)^{-1}C)^*=\overline{A-B(D-\overline{\lambda})^{-1}C}$$
for some (and hence for all) $\lambda\in\rho(D)$.
\item[(b)] If $A=A^*$ with $\mathcal D(A)\subset\mathcal D(B^*)$, then $\mathcal A$ is essentially self-adjoint if and only if $$(D-C(A-\lambda)^{-1}B)^*=\overline{D-C(A-\overline{\lambda})^{-1}B}$$
for some (and hence for all) $\lambda\in\rho(A)$.
\end{enumerate}
\end{theorem}

\begin{corollary}\label{corA3.3}
The following statements hold.
\begin{enumerate}
\item[(a)] Let $A$ be self-adjoint with $\mathcal D(|A|^\frac{1}{2})\subset\mathcal D(C)$
and let $D$ be essentially self-adjoint.
If $\mathcal D(B)\cap\mathcal D(D)$ is a core of $\overline{D}$,
then $\mathcal A$ is essentially self-adjoint.
\item[(b)] Let $A$ be essentially self-adjoint
and let $D$ be self-adjoint with $\mathcal D(|D|^\frac{1}{2})\subset\mathcal D(B)$.
If $\mathcal D(A)\cap\mathcal D(C)$ is a core of $\overline{A}$,
then $\mathcal A$ is essentially self-adjoint.
\end{enumerate}
\end{corollary}

\begin{proof}
We prove the claim in case (a); the proof in case (b) is analogous.
Let $\lambda\in\rho(A)$, we start to prove
 $$(D-C(A-\lambda)^{-1}B)^*=\overline{D-C(A-\overline{\lambda})^{-1}B}$$
via the arguments used in \cite[Section 1]{MS}.
Since $C\subset B^*$ and $\mathcal D(|A|^\frac{1}{2})\subset\mathcal D(C)$,
the operator
$$C(|A|+I)^{-\frac{1}{2}}=B^*(|A|+I)^{-\frac{1}{2}}$$
is closed and defined on the whole space.
It follows from the closed graph theorem that $C(|A|+I)^{-\frac{1}{2}}$ is bounded.
Consequently, the operator
$$(|A|+I)^{-\frac{1}{2}}B=(B^*(|A|+I)^{-\frac{1}{2}})^*|_{\mathcal D(B)}$$
is bounded on $\mathcal D(B)$.
Since the operator
$$U(\lambda):=(|A|+I)^\frac{1}{2}(A-\lambda)^{-1}(|A|+I)^\frac{1}{2}$$
is bounded on its domain $\mathcal D(|A|^\frac{1}{2})$,
the operator
$$C(A-\lambda)^{-1}B=C(|A|+I)^{-\frac{1}{2}}U(\lambda)(|A|+I)^{-\frac{1}{2}}B$$
is also bounded on its domain $\mathcal D(B)$.
Hence, if $\mathcal D(B)\cap\mathcal D(D)$ is a core of $\overline{D}$,
then by Lemma \ref{lemA3},
\begin{align*}
(D-C(A-\lambda)^{-1}B)^*&=D^*-(C(A-\lambda)^{-1}B)^*\\
                        &=\overline{D}-\overline{C(A-\overline{\lambda})^{-1}B}\\
                        &=\overline{D-C(A-\overline{\lambda})^{-1}B}.
\end{align*}
\end{proof}

\begin{remark}
It follows from Corollary \ref{corA3.3} and Theorem \ref{thA3.1} that
the linearized Navier-Stokes operator considered in \cite{FFMM}
is essentially self-adjoint and it is not closed.
\end{remark}

By the techniques used in the proofs of \cite[Theorem 3.1]{AJW} and the corresponding corollaries therein,
with some slight modifications,
we can prove the following theorem and related corollaries.

\begin{theorem}
Let $H_1=H_2$ and let $C=B^*$.
If $B$ is closed with $\rho(B)\neq\emptyset$
and $\mathcal D(\mathcal A)=\mathcal D(B^*)\times\mathcal D(B)$,
then the following statements are equivalent:
\begin{enumerate}
\item[(a)] $\mathcal A$ is self-adjoint,
\item[(b)] $(B-A(B^*-\overline{\lambda})^{-1}D)^*=B^*-D(B-\lambda)^{-1}A$ for some (and hence for all) $\lambda\in\rho(B)$,
\item[(c)] $(B^*-D(B-\lambda)^{-1}A)^*=B-A(B^*-\overline{\lambda})^{-1}D$ for some (and hence for all) $\lambda\in\rho(B)$.
\end{enumerate}
\end{theorem}

\begin{corollary}
Let $H_1=H_2$ and let $B=B^*=C$.
Then $\mathcal A$ is self-adjoint if one of the following holds:
\begin{enumerate}
\item[(a)] $A$ is $B$-bounded with relative bound $< 1$ and $D$ is $B$-bounded with relative bound $\leq 1$,
\item[(b)] $A$ is $B$-bounded with relative bound $\leq 1$ and $D$ is $B$-bounded with relative bound $< 1$.
\end{enumerate}
\end{corollary}

For the definition and properties of a maximal accretive operator in the following corollary, see \cite[Section IV.4]{Nagy}.
\begin{corollary}
Let $H_1=H_2$ and let $C=B^*$.
If $B$ or $-B$ is maximal accretive, then $\mathcal A$ is self-adjoint if one of the following holds:
\begin{enumerate}
\item[(a)] $A$ is $B^*$-bounded with relative bound $< 1$ and $D$ is $B$-bounded with relative bound $\leq 1$,
\item[(b)] $A$ is $B^*$-bounded with relative bound $\leq 1$ and $D$ is $B$-bounded with relative bound $< 1$.
\end{enumerate}
\end{corollary}

\begin{corollary}
Let $H_1=H_2$ and let $C=B^*$.
If $B$ is closed with $\rho(B)\neq\emptyset$ and
$\mathcal D(\mathcal A)=\mathcal D(B^*)\times\mathcal D(B)$,
then $\mathcal A$ is self-adjoint if one of the following holds:
\begin{enumerate}
\item[(a)] $A$ is $B^*$-bounded with relative bound $0$,
\item[(b)] $D$ is $B$-bounded with relative bound $0$.
\end{enumerate}
\end{corollary}

\appendix
\section{some lemmas on adjoints}

In this section, we will denote by $X,Y,Z$ Banach spaces.

Let $S,T$ be linear operators from $X$ to $Y$ and $X$ to $Z$, respectively.
Recall that $S$ is called $T$-bounded if $\mathcal D(T)\subset\mathcal D(S)$
and there exist constants $a,b\geq 0$ such that
$$\|Sx\|\leq a\|x\|+b\|Tx\|, x\in\mathcal D(T),$$
see \cite[Section VI.1.1]{Ka1980}.
The greatest lower bound
$b_0$ of all possible constants $b$ in the above inequality will be called the relative bound of $S$ with respective to $T$
(or simply the relative bound when there is no confusion).
If $T$ is closed and $S$ is closable, then $\mathcal D(T)\subset\mathcal D(S)$ already implies that $S$ is $T$-bounded (see \cite[Remark IV.1.5]{Ka1980}).

\begin{lemma}\label{lemA1}(\cite[Problem III.5.26]{Ka1980})
Let $S$ be a bounded everywhere defined operator from $Y$ to $Z$ and let $T$ be a densely defined operator from $X$ to $Y$.
Then $(ST)^*=T^*S^*$.
\end{lemma}

\begin{lemma}(\cite{S1970})\label{lemA2}
Let $S$ be a densely defined operator from $Y$ to $Z$ and let $T$ be a closed densely defined operator from $X$ to $Y$.
If the range $\mathcal R(T)$ of $T$ is closed in $Y$ and has finite codimension, then $ST$ is a
densely defined operator and $(ST)^*=T^*S^*$.
\end{lemma}

\begin{lemma}(\cite[Corollary 1]{HK1970})\label{lemA3}
Let $T$ be closed densely defined operators from $X$ to $Y$.
Suppose $S$ is a $T$-bounded operator such that $S^*$ is $T^*$-bounded, with both relative bounds $<1$.
Then $S+T$ is closed and $(S+T)^*=S^*+T^*$.
\end{lemma}

{\textbf{Acknowledgment.}}~{Project supported by Natural Science Foundation of China (Grant Nos. 11371185, 11361034)
and Major Subject of Natural Science Foundation of Inner Mongolia of China (Grant No. 2013ZD01).
The authors would like to express their sincere thanks to Tin-Yau Tam for his useful comments.}

\end{document}